\documentclass[english,12pt]{article}
\usepackage[T1]{fontenc}
\usepackage[utf8]{inputenc}
\usepackage[numbers]{natbib}
\usepackage{babel}
\usepackage{prettyref}
\usepackage{float}
\usepackage{units}
\usepackage{mathtools}
\usepackage{enumitem}
\usepackage{amsmath}
\usepackage{amsthm}
\usepackage{amssymb}
\usepackage{graphicx}
\usepackage[unicode=true]
{hyperref}

\makeatletter

\floatstyle{ruled}
\newfloat{algorithm}{tbp}{loa}
\providecommand{\algorithmname}{Algorithm}
\floatname{algorithm}{\protect\algorithmname}

\numberwithin{equation}{section}
\numberwithin{figure}{section}
\theoremstyle{plain}
\newtheorem{thm}{\protect\theoremname}[section]
\theoremstyle{definition}
\newtheorem{defn}[thm]{\protect\definitionname}
\theoremstyle{definition}
\newtheorem{example}[thm]{\protect\examplename}
\theoremstyle{plain}
\newtheorem{lem}[thm]{\protect\lemmaname}
\ifx\proof\undefined
\newenvironment{proof}[1][\protect\proofname]{\par
  \normalfont\topsep6\p@\@plus6\p@\relax
  \trivlist
  \itemindent\parindent
\item[\hskip\labelsep\scshape #1]\ignorespaces
}{%
  \endtrivlist\@endpefalse
}
\providecommand{\proofname}{Proof}
\fi
\theoremstyle{plain}
\newtheorem{cor}[thm]{\protect\corollaryname}
\theoremstyle{plain}
\newtheorem{prop}[thm]{\protect\propositionname}
\theoremstyle{remark}
\newtheorem{rem}[thm]{\protect\remarkname}
\usepackage{algolyx}

\usepackage[dvipsnames]{xcolor}
\definecolor{mylinkcolor}{HTML}{0066cc}
\definecolor{mycitecolor}{HTML}{008800}
\definecolor{myurlcolor}{HTML}{0066cc}

\usepackage{multicol}
\usepackage{tikz}
\usetikzlibrary{arrows,decorations.pathreplacing,shapes}
\tikzset{node distance=1.5cm}
\tikzset{graph node/.style={circle, fill=gray!30}}
\tikzset{graph edge/.style={thick, >=stealth'}}

\DeclareMathOperator*{\infd}{inf\vphantom{\operator@font p}}

\@ifclassloaded{amsart}{
  \AtBeginDocument{%
    \def\MR#1{}
  }

  \copyrightinfo{\currentyear}{American Mathematical Society}

}{
  \pdfoutput=1

  \tolerance=10000

  \usepackage{lmodern}
  \usepackage[margin=1in]{geometry}
  \usepackage{microtype}

}

\usepackage{hyperref}
\hypersetup{
  linkcolor  = mylinkcolor,
  citecolor  = mycitecolor,
  urlcolor   = myurlcolor,
  colorlinks = true
}
\providecommand{\doi}{}
\renewcommand{\doi}[1]{\href{https://doi.org/#1}{doi:
\begingroup\urlstyle{rm}\nolinkurl{#1}\endgroup}}

\@ifundefined{showcaptionsetup}{}{%
\PassOptionsToPackage{caption=false}{subfig}}
\usepackage{subfig}
\makeatother

\providecommand{\corollaryname}{Corollary}
\providecommand{\definitionname}{Definition}
\providecommand{\examplename}{Example}
\providecommand{\lemmaname}{Lemma}
\providecommand{\propositionname}{Proposition}
\providecommand{\remarkname}{Remark}
\providecommand{\theoremname}{Theorem}

\begin{document}
% References
\newrefformat{subsec}{Section \ref{#1}}
\newrefformat{assu}{Assumption \ref{#1}}
\newrefformat{cor}{Corollary \ref{#1}}
\newrefformat{def}{Definition \ref{#1}}
\newrefformat{exa}{Example \ref{#1}}
\newrefformat{prop}{Proposition \ref{#1}}
\newrefformat{rem}{Remark \ref{#1}}
\newrefformat{alg}{Algorithm \ref{#1}}
\newrefformat{fig}{Figure \ref{#1}}
\newrefformat{app}{Appendix \ref{#1}}

% Black box for proof
\renewcommand{\qedsymbol}{$\blacksquare$}

\global\long\def\diag{\operatorname{diag}}

\global\long\def\conn{\operatorname{con}}

\global\long\def\graph{\operatorname{graph}}

\global\long\def\walks{\operatorname{walks}}

\global\long\def\head{\operatorname{head}}

\global\long\def\last{\operatorname{last}}

\global\long\def\proj{\operatorname{proj}}

\date{}

\title{A fast and stable test to check if a weakly diagonally dominant matrix
is a nonsingular M-matrix}

\author{Parsiad Azimzadeh\thanks{David R. Cheriton School of Computer Science,
    University of Waterloo, Waterloo ON, Canada N2L 3G1
\href{mailto:pazimzad@uwaterloo.ca}{pazimzad@uwaterloo.ca}.}}\maketitle

\begin{abstract}
  We present a test for determining if a substochastic matrix is convergent.
  By establishing a duality between weakly chained diagonally dominant
  (w.c.d.d.) L-matrices and convergent substochastic matrices, we show
  that this test can be trivially extended to determine whether a weakly
  diagonally dominant (w.d.d.) matrix is a nonsingular M-matrix. The
  test's runtime is linear in the order of the input matrix if it is
  sparse and quadratic if it is dense. This is a partial strengthening
  of the cubic test in {[}J. M. Peña., \emph{A stable test to check
    if a matrix is a nonsingular M-matrix}, Math. Comp., 247,
    1385\textendash 1392,
  2004{]}. As a by-product of our analysis, we prove that a nonsingular
  w.d.d. M-matrix is a w.c.d.d. L-matrix, a fact whose converse has
  been known since at least 1964. We point out that this strengthens
  some recent results on M-matrices in the literature.
\end{abstract}

\section{\label{sec:introduction}Introduction}

The \emph{substochastic matrices} are real matrices with nonnegative
entries and whose row-sums are at most one. We establish two results
relating to this family:
\begin{enumerate}[label=(\emph{\roman*})]
  \item \label{enu:results_1}To each substochastic matrix $B$ we associate
    a possibly infinite \emph{index of contraction} $\widehat{\conn}B$
    and show that for each nonnegative integer $k$, $B^{k}$ is a contraction
    in the infinity norm (i.e., $\Vert B^{k}\Vert_{\infty}<1$) if and
    only if $k>\widehat{\conn}B$.
  \item \label{enu:results_2}We show that the index of contraction of a sparse
    (resp. dense) square substochastic matrix is computable in time linear
    (resp. quadratic) in the order of the input matrix.
\end{enumerate}
It follows immediately from \ref{enu:results_1} that a square substochastic
matrix is convergent if and only if its index of contraction is finite.

By establishing a duality between \emph{weakly chained diagonally
dominant} (w.c.d.d.) \emph{L-matrices} and convergent substochastic
matrices, we use point \ref{enu:results_2} to obtain a test to determine
whether a weakly diagonally dominant (w.d.d.) matrix is a nonsingular
\emph{M-matrix}. Previous work in this regard is the test in \cite{MR2047092}
to determine if an arbitrary matrix (not necessarily w.d.d.) is a
nonsingular M-matrix, which has a cost asymptotically equivalent to
Gaussian elimination (i.e., cubic in the order of the input matrix).

W.d.d. M-matrices arise naturally from discretizations of differential
operators and appear in the Bellman equation for optimal decision
making on a controlled Markov chain \cite{MR2551155}. As such, these
matrices have attracted a significant amount of attention from the
scientific computing and numerical analysis communities.

W.c.d.d. matrices were first studied in a wonderful work by P. N.
Shivakumar and K. H. Chew \cite{MR0332820} in which they were proven
to be nonsingular (see also \cite{MR3493959} for a short proof).
Various authors have recently studied the family of w.c.d.d. M-matrices,
obtaining bounds on the infinity norm of their inverses (i.e., $\Vert
A^{-1}\Vert_{\infty}$)
\cite{MR1384509,MR2350685,MR2433738,MR2535528,MR2577710}. While a
w.c.d.d. matrix is w.d.d. by definition, the converse is not necessarily
true in general (e.g., $\left(
    \begin{smallmatrix} +1&-1\\ -1&+1
  \end{smallmatrix} \right)$
is w.d.d. but not w.c.d.d.).

It has long been known (possibly as early as 1964; see the work of
J. H. Bramble and B. E. Hubbard \cite{MR0162367}) that a w.c.d.d.
L-matrix\footnote{In \cite{MR0162367}, the authors refer to w.c.d.d.
  L-matrices as
\emph{matrices of positive type}.} is a nonsingular w.d.d. M-matrix.
We obtain a proof of the converse
as a by-product of our analysis. In particular, we establish
that\footnote{\eqref{eq:characterization} remains true if we replace
  ``L-matrix''
by ``\emph{Z-matrix} with nonnegative diagonal entries''.}
\begin{align}
  A\text{ is a nonsingular w.d.d. M-matrix} & \iff A\text{ is a
  nonsingular w.d.d. L-matrix}\nonumber \\
  & \iff A\text{ is a w.c.d.d. L-matrix}.\label{eq:characterization}
\end{align}
\eqref{eq:characterization} immediately strengthens the results pertaining
to norms of inverses listed in the previous paragraph, ensuring they
apply more generally to nonsingular w.d.d. M-matrices.
\eqref{eq:characterization}
is also useful in that it gives a graph-theoretic characterization
of nonsingular w.d.d. M-matrices by means of w.c.d.d. matrices. This
characterization is often easier to use than the usual characterizations
involving, say, inverse-positivity or positive principal minors
\cite{MR0444681}.

We list a few other interesting recent results concerning w.c.d.d.
matrices and M-matrices here:
\cite{MR2653540,MR3175363,MR3149442,MR3253625,MR3357017,li2016subdirect,MR3579726,MR3486018}.

\prettyref{sec:matrix_families} introduces and establishes results
on substochastic matrices, M-matrices, and w.c.d.d. matrices.
\prettyref{sec:computation}
gives the procedure to compute the index of contraction.
\prettyref{sec:experiments}
presents numerical experiments testing the efficacy of the procedure
on randomly sampled matrices.

\section{\label{sec:matrix_families}Matrix families}

\subsection{Substochastic matrices}
\begin{defn}
  A substochastic matrix is a real matrix $B\coloneqq(b_{ij})$ with
  nonnegative entries (i.e., $b_{ij}\geq0$) and row-sums at most one
  (i.e., $\sum_{j}b_{ij}\leq1$). A stochastic (a.k.a. Markov) matrix
  is a substochastic matrix whose row-sums are exactly one.
\end{defn}
Note that in our definition above, we do not require $B$ to be square.
\begin{defn}
  Let $A\coloneqq(a_{ij})$ be an $m\times n$ complex matrix.
  \begin{enumerate}[label=(\emph{\roman*})]
    \item The digraph of $A$, denoted $\graph A$, is defined as follows:
      \begin{enumerate}
        \item If $A$ is square, $\graph A$ is a tuple $(V,E)$ consisting of the
          vertex set $V\coloneqq\{1,\ldots,m\}$ and edge set
          $E\subset V\times V$
          satisfying $(i,j)\in E$ if and only if $a_{ij}\neq0$.
        \item If $A$ is not square, $\graph A\coloneqq\graph A^{\prime}$ where
          $A^{\prime}$ is the smallest square matrix obtained by appending
          rows or columns of zeros to $A$.
      \end{enumerate}
    \item A walk in $\graph A\equiv(V,E)$ is a nonempty finite sequence of
      edges $(i_{1},i_{2})$, $(i_{2},i_{3})$, $\ldots$, $(i_{\ell-1},i_{\ell})$
      in $E$. The set of all walks in $\graph A$ is denoted $\walks A$.
    \item Let $p\in\walks A$. The length of $p$, denoted $|p|$, is the total
      number of edges in $p$. $\head p$ (resp. $\last p$) is the first
      (resp. last) vertex in $p$.
  \end{enumerate}
\end{defn}
To simplify matters, we hereafter denote edges by $i\rightarrow j$
instead of $(i,j)$ and walks by $i_{1}\rightarrow
i_{2}\rightarrow\cdots\rightarrow i_{\ell}$
instead of $(i_{1},i_{2})$, $(i_{2},i_{3})$, $\ldots$, $(i_{\ell-1},i_{\ell})$.
We use the terms ``row'' and ``vertex'' interchangeably.

\begin{figure}
  \begin{centering}
    \subfloat[The matrix $B$]{
      \begin{centering}
        $
        \begin{pmatrix}0\\
          1 & 0\\
          & 1 & 0\\
          &  & \ddots & \ddots\\
          &  &  & 1 & 0
        \end{pmatrix}$
        \par
      \end{centering}
    }\hfill{}\subfloat[$\protect\graph B$ (vertices in $\hat{J}(B)$
    are \colorbox{mylinkcolor}{\textcolor{white}{highlighted}})]{
      \begin{tikzpicture}[node distance=2cm]
        \node [graph node, fill=mylinkcolor, text=white] (1) {1};
        \node [graph node, right of=1] (2) {2};
        \node [right of=2] (dots) {$\cdots$};
        \node [graph node, right of=dots] (n) {$n$};
        \draw[graph edge, ->] (2) to (1);
        \draw[graph edge, ->] (dots) to (2);
        \draw[graph edge, ->] (n) to (dots);
      \end{tikzpicture}

    }
    \par
  \end{centering}
  \caption{\label{fig:submarkov_example}An example of an $n\times n$
    substochastic
  matrix and its graph}
\end{figure}
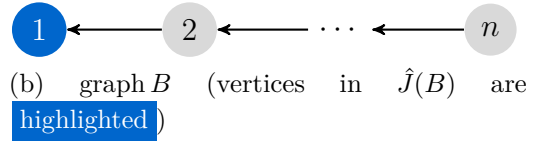

Let $B\coloneqq(b_{ij})$ be an $m\times n$ substochastic matrix.
We define the sets
\begin{align*}
  \hat{J}(B) & \coloneqq\left\{ 1\leq i\leq m\colon{\textstyle
  \sum_{j}}b_{ij}<1\right\} \\
  \text{and }\hat{P}_{i}(B) & \coloneqq\Bigl\{ p\in\walks
  B\colon\head p=i\text{ and }\last p\in\hat{J}(B)\Bigr\}.
\end{align*}
It is understood that when we write $i\notin\hat{J}(B)$, we mean
$i\in\hat{J}(B)^{\complement}\coloneqq\{1,\ldots,m\}\setminus\hat{J}(B)$.
Note that if $\hat{J}(B)$ is empty, so too is $\hat{P}_{i}(B)$ for
each $i$. We define the index of contraction associated with $B$
by
\begin{equation}
  \widehat{\conn}B\coloneqq\max\left(0,\sup_{i\notin\hat{J}(B)}\left\{
    \infd_{p\in\hat{P}_{i}(B)}\left|p\right|\right\}
  \right)\label{eq:index_of_connectivity}
\end{equation}
subject to the conventions $\inf\emptyset=\infty$ and $\sup\emptyset=-\infty$.
We will see shortly that a square matrix $B$ is convergent if and only
if $\widehat{\conn}B$ is finite.
\begin{example}
  The $n\times n$ matrix $B$ in \prettyref{fig:submarkov_example}
  satisfies $\hat{J}(B)=\{1\}$ and
  \[
    \min_{p\in\hat{P}_{i}(B)}\left|p\right|=i-1\text{ for }i\notin\hat{J}(B).
  \]
  It follows that $\widehat{\conn}B=n-1$.
\end{example}
An immediate consequence of the definition of the index of contraction
is below.
\begin{lem}
  \label{lem:bounded_or_infinite}Let $B$ be an $m\times n$ substochastic
  matrix. If $m\leq n$ (resp. $m>n$) $\widehat{\conn}B$ is either
  infinite or strictly less than $m$ (resp. $n+1$).
\end{lem}
\begin{proof}
  Suppose $m\leq n$. Let $i_{1}\notin\hat{J}(B)$ and $p\coloneqq
  i_{1}\rightarrow\cdots\rightarrow i_{\ell}$
  be a walk in $\hat{P}_{i_{1}}(B)$. Since $i_{\ell}\in\hat{J}(B)$,
  it follows that $1\leq i_{\ell}\leq m$. This implies that $1\leq i_{k}\leq m$
  for all $k$ since by definition, $\graph B$ has no edges of the
  form $i\rightarrow j$ where $i>m$. Now, suppose $|p|\geq m$. By
  the pigeonhole principle, we can find integers $u$ and $v$ such
  that $1\leq u<v\leq\ell$ and $i_{u}=i_{v}$. That is, the walk $p$
  contains a cycle (i.e., a subwalk starting and ending at the same
  vertex). ``Removing'' the cycle yields the new walk
  \[
    p^{\prime}\coloneqq i_{1}\rightarrow
    i_{2}\rightarrow\cdots\rightarrow i_{u}\rightarrow
    i_{v+1}\rightarrow i_{v+2}\rightarrow\cdots\rightarrow i_{\ell}
  \]
  in $\hat{P}_{i_{1}}(B)$ satisfying $|p^{\prime}|<|p|$. If
  $|p^{\prime}|\geq m$,
  we can continue removing cycles until we arrive at a walk
  $p^{\prime\prime}\in\hat{P}_{i_{1}}(B)$
  satisfying $|p^{\prime\prime}|<m$.

  The case of $m>n$ is handled similarly.
\end{proof}
We are now ready to present our main result related to substochastic
matrices. In the statement below, it is understood that if $B$ is
a square matrix, $B^{0}=I$.
\begin{thm}
  \label{thm:eventual_contractivity}Let $B$ be a square substochastic
  matrix. If $\alpha\coloneqq\widehat{\conn}B$ is finite,
  \[
    1=\Vert B^{0}\Vert_{\infty}=\cdots=\Vert
    B^{\alpha}\Vert_{\infty}>\Vert
    B^{\alpha+1}\Vert_{\infty}\geq\Vert B^{\alpha+2}\Vert_{\infty}\geq\cdots
  \]
  Otherwise,
  \[
    1=\Vert B^{1}\Vert_{\infty}=\Vert B^{2}\Vert_{\infty}=\cdots
  \]
\end{thm}
Before giving a proof, it is useful to record some consequences of
the above.
\begin{cor}
  \label{cor:eventual_contractivity}Let $B$ be a square substochastic
  matrix. Then, its spectral radius is no larger than one. Moreover,
  the following statements are equivalent:
  \begin{enumerate}[label=(\roman*), ref=(\emph{\roman*})]
    \item \label{enu:eventual_contractivity_1}$\widehat{\conn}B$ is finite.
    \item \label{enu:eventual_contractivity_2}$B$ is convergent.
    \item \label{enu:eventual_contractivity_3}$I-B$ is nonsingular.
  \end{enumerate}
\end{cor}
The above can be considered a generalization of the well-known result
that a square stochastic (a.k.a. Markov) matrix has spectral radius
no larger than one and at least one eigenvalue equal exactly to one
(recall that for any matrix $M$, $I-M$ is singular if and only
if $\lambda=1$ is an eigenvalue of $M$).
\begin{proof}
  The claim that the spectral radius of $B$ is no larger than one in
  magnitude is a direct consequence of the fact that $\Vert
  B\Vert_{\infty}\leq1$.

  \ref{enu:eventual_contractivity_1} $\implies$
  \ref{enu:eventual_contractivity_2}
  follows immediately from \prettyref{thm:eventual_contractivity},
  while \ref{enu:eventual_contractivity_2} $\implies$
  \ref{enu:eventual_contractivity_3}
  is true for any matrix. We prove below, by contrapositive, the claim
  \ref{enu:eventual_contractivity_3} $\implies$
  \ref{enu:eventual_contractivity_1}.

  Suppose $\widehat{\conn}B$ is infinite. Let $R$ be the set of rows
  $i\notin\hat{J}(B)$ for which $\hat{P}_{i}(B)$ is empty. Due to
  our assumptions, there is at least one such row and hence $R$ is
  nonempty. Without loss of generality, we may assume $R=\{1,\ldots,r\}$
  for some $1\leq r\leq n$ where $n$ is the order of $B$ (otherwise,
    replace $B$ by $PBP^{\intercal}$ where $P$ is an appropriately
  chosen permutation matrix). Let $e\in\mathbb{R}^{r}$ be the column
  vector whose entries are all one. If $r=n$, each row-sum of $B$
  is one (i.e., $Be=e$ so that $(I-B)e=0$). Otherwise, $B$ has the
  block structure\[
    B = \left(
      \begin{array}{c|c}
        B_1 & 0 \\
        \hline
        B_2 & B_3
    \end{array} \right)
    \text{ where } B_1 \in \mathbb{R}^{r\times r}.
  \]The partition above ensures that for each row $i\notin R$, $i\in\hat{J}(B)$
  or $\hat{P}_{i}(B)$ is nonempty. Therefore, $\widehat{\conn}B_{3}$
  is finite, and hence the linear system $(I-B_{3})x=B_{2}e$ has a
  unique solution $x$. Moreover, since the row-sums of $B_{1}$ are
  one, $B_{1}e=e$. Therefore,
  \[
    (I-B)
    \begin{pmatrix}e\\
      x
    \end{pmatrix}=
    \begin{pmatrix}e\\
      x
    \end{pmatrix}-
    \begin{pmatrix}B_{1}e\\
      B_{2}e+B_{3}x
    \end{pmatrix}=
    \begin{pmatrix}e\\
      x
    \end{pmatrix}-
    \begin{pmatrix}e\\
      x
    \end{pmatrix}=0.\qedhere
  \]
\end{proof}
\begin{cor}
  \label{cor:irreducibly_substochastic}A square irreducible substochastic
  matrix $B$ is convergent if and only if $\hat{J}(B)$ is nonempty.
\end{cor}
The above result is well-known. It can be obtained, for example, by
\cite[Corollary 1.19 and Lemma 2.8]{MR1753713}. We give a short alternate
proof using \prettyref{cor:eventual_contractivity}:
\begin{proof}
  Since a square matrix is irreducible if and only if its digraph is
  strongly connected \cite{MR1753713}, $\widehat{\conn}B$ is finite
  if and only if $\hat{J}(B)$ is nonempty. The result now follows from
  \prettyref{cor:eventual_contractivity}.
\end{proof}
If $B$ is a square substochastic matrix, we can always find a permutation
matrix $P$ and an integer $r\geq1$ such that $PBP^{\intercal}$
has the block triangular structure
\begin{equation}
  PBP^{\intercal}=
  \begin{pmatrix}B_{11} & B_{12} & \cdots & B_{1r}\\
    & B_{22} & \cdots & B_{2r}\\
    &  & \ddots & \vdots\\
    &  &  & B_{rr}
  \end{pmatrix}\label{eq:normal_form}
\end{equation}
where each $B_{ii}$ is a square substochastic matrix that is either
irreducible or a $1\times1$ zero matrix (it is understood that if
$r=1$, then $B=B_{11}$). Following \cite{MR0107648,MR1753713},
we refer to this as the \emph{normal form} of $B$ (it is shown in
  \cite[Pg. 90]{MR0107648} that the normal form of a matrix is unique
up to permutations by blocks). Since $\det(PBP^{\intercal}-\lambda
I)=\prod_{i}\det(B_{ii}-\lambda I)$,
the spectrum of $B$ satisfies
\begin{equation}
  \sigma(B)=\sigma(B_{11})\cup\cdots\cup\sigma(B_{rr}).\label{eq:spectrum}
\end{equation}
This observation motivates the next result.
\begin{thm}
  Let $B$ be a square substochastic matrix with normal form
  \eqref{eq:normal_form}.
  $B$ is convergent if and only if $\hat{J}(B_{ii})$ is nonempty for
  each $i$. Moreover, if $B$ is convergent,
  \begin{equation}
    \max_{i}\left\{ \widehat{\conn}B_{ii}\right\}
    \leq\widehat{\conn}B\leq N+\widehat{\conn}B_{rr}\label{eq:conn_inequality}
  \end{equation}
  where $N\coloneqq\sum_{i=1}^{r-1}n_{i}$ and $n_{i}$ is the order
  of the matrix $B_{ii}$ (it is understood that if $r=1$, then $N=0$).
\end{thm}
\begin{proof}
  The first claim is a consequence of \prettyref{cor:irreducibly_substochastic}
  and \eqref{eq:spectrum}.

  We prove now the leftmost inequality in \eqref{eq:conn_inequality}.
  First, note that $\Vert B^{k}\Vert_{\infty}=\Vert
  PB^{k}P^{\intercal}\Vert_{\infty}=\Vert(PBP^{\intercal})^{k}\Vert_{\infty}$.
  Moreover, the block diagonal entries of $(PBP^{\intercal})^{k}$ are
  the matrices $B_{11}^{k},\ldots,B_{rr}^{k}$. Therefore, for each
  $i$, $\Vert B_{ii}^{k}\Vert_{\infty}\leq\Vert B^{k}\Vert_{\infty}$
  and hence $\widehat{\conn}B_{ii}\leq\widehat{\conn}B$ by
  \prettyref{thm:eventual_contractivity}.

  We prove now the rightmost inequality in \eqref{eq:conn_inequality}.
  If $\widehat{\conn}B\leq N$, the inequality is trivial. As such,
  we proceed assuming that $N<\widehat{\conn}B<\infty$. First, note
  that $\widehat{\conn}B=\widehat{\conn}(PBP^{\intercal})$. Therefore,
  $\widehat{\conn}B=|p|$ where $p$ is a walk whose length is no larger
  than any walk in $\hat{P}_{i_{1}}(PBP^{\intercal})$ and
  $i_{1}\coloneqq\head p$.
  Due to the block triangular structure of $PBP^{\intercal}$, we can
  write $p$ as
  \[
    p=i_{1}\rightarrow\cdots\rightarrow i_{u}\rightarrow
    j_{1}\rightarrow\cdots\rightarrow j_{v}
  \]
  where $u\leq N$ and $j_{k}>N$ for all $k$. Defining
  $j_{k}^{\prime}\coloneqq j_{k}-N$,
  it follows that $p^{\prime}\coloneqq
  j_{1}^{\prime}\rightarrow\cdots\rightarrow j_{v}^{\prime}$
  is a walk whose length is no larger than any walk in
  $\hat{P}_{j_{1}^{\prime}}(B_{rr})$,
  from which we obtain $|p^{\prime}|\leq\widehat{\conn}B_{rr}$. Therefore,
  \[
    \widehat{\conn}B=\left|p\right|\leq u+\left|p^{\prime}\right|\leq
    N+\widehat{\conn}B_{rr}.\qedhere
  \]
\end{proof}
Returning to our goal of proving \prettyref{thm:eventual_contractivity},
we first establish some lemmata related to substochastic matrices.
The first lemma is a consequence of definitions and requires no proof.
\begin{lem}
  \label{lem:contraction_characterization}Let $B$ be an $m\times n$
  substochastic matrix. Then, $\Vert B\Vert_{\infty}<1$ if and only
  if $\hat{J}(B)=\{1,\ldots,m\}$.
\end{lem}
\begin{lem}
  \label{lem:BC_results}Let $B\coloneqq(b_{ij})$ and $C\coloneqq(c_{ij})$
  be compatible (i.e., the product $BC$ is well-defined) substochastic
  matrices. Then,
  \begin{enumerate}[label=(\roman*), ref=(\emph{\roman*})]
    \item \label{enu:BC_results_1}$BC$ is a substochastic matrix.
    \item \label{enu:BC_results_2}If $i\in\hat{J}(B)$, then $i\in\hat{J}(BC)$.
    \item \label{enu:BC_results_3}If $i\notin\hat{J}(B)$, then $i\in\hat{J}(BC)$
      if and only if there exists $h\in\hat{J}(C)$ such that $i\rightarrow h$
      is an edge in $\graph B$.
    \item \label{enu:BC_results_4}$i\rightarrow j$ is an edge in $\graph(BC)$
      if and only if there exist edges $i\rightarrow h$ and $h\rightarrow j$
      in $\graph B$ and $\graph C$, respectively.
  \end{enumerate}
\end{lem}
\begin{proof}
  ~
  \begin{enumerate}[label=(\emph{\roman*})]
    \item $BC$ has nonnegative entries and $\Vert
      BCe\Vert_{\infty}\leq\Vert BC\Vert_{\infty}\leq\Vert
      B\Vert_{\infty}\Vert C\Vert_{\infty}\leq1$.
    \item Note first that
      $\sum_{j}[BC]_{ij}=\sum_{j}\sum_{k}b_{ik}c_{kj}=\sum_{k}b_{ik}\sum_{j}c_{kj}\leq\sum_{k}b_{ik}$.
      If $i\in\hat{J}(B)$, then $\sum_{k}b_{ik}<1$ and the desired result
      follows.
    \item Suppose $i\notin\hat{J}(B)$. If there exists $h\in\hat{J}(C)$ such
      that $i\rightarrow h$ is an edge in $\graph B$, then $\sum_{j}c_{hj}<1$
      and $\sum_{j}[BC]_{ij}=b_{ih}\sum_{j}c_{hj}+\sum_{k\neq
      h}b_{ik}\sum_{j}c_{kj}<\sum_{k}b_{ik}\leq1$.
      Otherwise, $\sum_{j}c_{kj}=1$ for all $k$ with $b_{ik}\neq0$ and
      hence $\sum_{j}[BC]_{ij}=\sum_{k}b_{ik}\sum_{j}c_{kj}=\sum_{k}b_{ik}=1$.
    \item Suppose $i\rightarrow h$ and $h\rightarrow j$ are edges in $\graph B$
      and $\graph C$, respectively. Then,
      $[BC]_{ij}=\sum_{k}b_{ik}c_{kj}\geq b_{ih}c_{hj}>0$.
      Otherwise, for each $k$, at least one of $b_{ik}$ or $c_{kj}$ is
      zero and hence $[BC]_{ij}=0$.\hfill\qedhere
  \end{enumerate}
\end{proof}
\begin{lem}
  \label{lem:shrinking_paths}Let $B$ be a square substochastic matrix,
  $i\notin\hat{J}(B)$, and $k$ be a positive integer. Then, $i\in\hat{J}(B^{k})$
  if and only if there is a walk $p$ in $\hat{P}_{i}(B)$ such that
  $|p|<k$.
\end{lem}
\begin{proof}
  To simplify notation, let $i_{1}\coloneqq i$.

  Suppose there exists a walk $i_{1}\rightarrow
  i_{2}\rightarrow\cdots\rightarrow i_{\ell}$
  in $\hat{P}_{i_{1}}(B)$. We claim that $i_{1}\rightarrow i_{\ell}$
  is an edge in $\graph(B^{\ell-1})$. If this is the case,
  \prettyref{lem:BC_results}
  \ref{enu:BC_results_2} and \ref{enu:BC_results_3} guarantee that
  $i_{1}\in\hat{J}(B^{\ell-1}B)=\hat{J}(B^{\ell})$. If $\ell\leq k$,
  $i_{1}\in\hat{J}(B^{k})$ by \prettyref{lem:BC_results} \ref{enu:BC_results_2},
  as desired.

  We now return to the claim in the previous paragraph. Since the claim
  is trivial if $\ell=2$, we proceed assuming $\ell>2$. Let $n$ be
  an integer satisfying $2<n\leq\ell$. If $i_{1}\rightarrow i_{n-1}$
  is an edge in $\graph(B^{n-2})$, then since $i_{n-1}\rightarrow i_{n}$
  is an edge in $\graph B$, \prettyref{lem:BC_results} \ref{enu:BC_results_4}
  implies that $i_{1}\rightarrow i_{n}$ is an edge in
  $\graph(B^{n-2}B)=\graph(B^{n-1})$.
  Since $i_{1}\rightarrow i_{2}$ is an edge in $\graph B$, it follows
  by induction that $i_{1}\rightarrow i_{\ell}$ is an edge in
  $\graph(B^{\ell-1})$,
  as desired.

  As for the converse,  suppose $i_{1}\in\hat{J}(B^{k})$. Let $\ell$
  be the smallest positive integer such that $i_{1}\notin\hat{J}(B^{\ell-1})$
  and $i_{1}\in\hat{J}(B^{\ell})$. Since $i_{1}\notin\hat{J}(B)$ and
  $i_{1}\in\hat{J}(B^{k})$, it follows that $\ell\leq k$. By
  \prettyref{lem:BC_results}
  \ref{enu:BC_results_3}, there exists $i_{\ell}\in\hat{J}(B)$ such
  that $i_{1}\rightarrow i_{\ell}$ is an edge in $\graph(B^{\ell-1})$.

  If $\ell=2$, the trivial walk $i_{1}\rightarrow i_{\ell}$ is in
  $\hat{P}_{i_{1}}(B)$, and hence we proceed assuming $\ell>2$. Let
  $n$ be an integer satisfying $2<n\leq\ell$. If there exists a positive
  integer $i_{n}$ such that $i_{1}\rightarrow i_{n}$ is an edge in
  $\graph(B^{n-1})=\graph(B^{n-2}B)$, \prettyref{lem:BC_results}
  \ref{enu:BC_results_4}
  implies that there exists a positive integer $i_{n-1}$ such that
  $i_{1}\rightarrow i_{n-1}$ is an edge in $\graph(B^{n-2})$ and
  $i_{n-1}\rightarrow i_{n}$
  is an edge in $\graph B$. Since $i_{1}\rightarrow i_{\ell}$ is an
  edge in $\graph(B^{\ell-1})$, it follows by induction that
  $i_{n-1}\rightarrow i_{n}$
  is an edge in $\graph B$ for each integer $n$ satisfying $2\leq n\leq\ell$.
  Therefore, $i_{1}\rightarrow i_{2}\rightarrow\cdots\rightarrow i_{\ell}$
  is a walk in $\hat{P}_{i_{1}}(B)$, as desired.
\end{proof}
We are now ready to prove \prettyref{thm:eventual_contractivity}.
\begin{proof}[Proof of \prettyref{thm:eventual_contractivity}]
  Since $\Vert B^{k+1}\Vert_{\infty}\leq\Vert
  B^{k}\Vert_{\infty}\Vert B\Vert_{\infty}\leq\Vert B^{k}\Vert_{\infty}$,
  the inequalities $1\geq\Vert B^{1}\Vert_{\infty}\geq\Vert
  B^{2}\Vert_{\infty}\geq\cdots$
  follow trivially.

  The remaining inequalities in the theorem statement follow by applying
  \prettyref{lem:shrinking_paths} to each row not in $\hat{J}(B)$
  and invoking \prettyref{lem:contraction_characterization}.
\end{proof}

\subsection{M-matrices}

In this subsection, we recall some well-known results on M-matrices
(see, e.g., \cite[Chapter 6]{MR1298430}).
\begin{defn}
  An M-matrix is a square matrix $A$ that can be expressed in the form
  $A=sI-B$ where $B$ is a nonnegative matrix and $s\geq\rho(B)$ where
  $\rho(B)$ is the spectral radius of $B$.
\end{defn}
\begin{defn}
  A Z-matrix is a real matrix with nonpositive off-diagonal entries.
\end{defn}
\begin{defn}
  \label{def:L_matrix}An L-matrix is a Z-matrix with positive diagonal
  entries.
\end{defn}
\begin{prop}
  \label{prop:an_M_matrix_is_an_L_matrix}A nonsingular M-matrix is
  an L-matrix.
\end{prop}
\begin{defn}
  \label{def:collatz}Let $A$ be a square real matrix. $A$ is monotone
  if and only if it is nonsingular and its inverse consists only of
  nonnegative entries.
\end{defn}
\begin{prop}
  \label{prop:M_matrix_monotone}The following are equivalent:
  \begin{enumerate}[label=(\roman*), ref=(\emph{\roman*})]
    \item $A$ is a nonsingular M-matrix.
    \item $A$ is a monotone Z-matrix.
  \end{enumerate}
\end{prop}
We close this subsection by introducing the following enlargement
of the family of L-matrices (\prettyref{def:L_matrix}), to be used
in the sequel.
\begin{defn}
  An $\operatorname{L}_{0}$-matrix is a Z-matrix with \emph{nonnegative}
  diagonal entries.
\end{defn}

\subsection{Weakly chained diagonally dominant (w.c.d.d.) matrices}

Before we can define w.c.d.d. matrices, we require some preliminary
definitions.
\begin{defn}
  Let $A\coloneqq(a_{ij})$ be a square complex matrix.
  \begin{enumerate}[label=(\emph{\roman*})]
    \item The $i$-th row of $A$ is w.d.d. (resp. s.d.d.) if
      $|a_{ii}|\geq\sum_{j\neq i}|a_{ij}|$
      (resp. $>$).
    \item $A$ is w.d.d. (resp. s.d.d.) if all of its rows are w.d.d. (resp.
      s.d.d.).
  \end{enumerate}
\end{defn}
Let $A\coloneqq(a_{ij})$ be an $m\times m$ complex w.d.d. matrix.
We define the sets
\begin{align*}
  J(A) & \coloneqq\left\{ 1\leq i\leq
    m\colon\left|a_{ii}\right|>{\textstyle \sum_{j\neq
  i}}\left|a_{ij}\right|\right\} \\
  \text{ and }P_{i}(A) & \coloneqq\Bigl\{ p\in\walks A\colon\head
  p=i\text{ and }\last p\in J(A)\Bigr\}.
\end{align*}
Note that if $J(A)$ is empty, so too is $P_{i}(A)$ for each $i$.
We will see shortly that the sets $J(\cdot)$ and $P_{i}(\cdot)$
are related to $\hat{J}(\cdot)$ and $\hat{P}_{i}(\cdot)$.

We are now ready to introduce w.c.d.d. matrices:
\begin{defn}
  \label{def:wcdd}A square complex matrix $A$ is w.c.d.d. if the points
  below are satisfied:
  \begin{enumerate}[label=(\emph{\roman*})]
    \item \label{enu:wcdd_1}$A$ is w.d.d.
    \item \label{enu:wcdd_extra}$J(A)$ is nonempty.
    \item \label{enu:wcdd_2}For each $i\notin J(A)$, $P_{i}(A)$ is nonempty.
  \end{enumerate}
\end{defn}

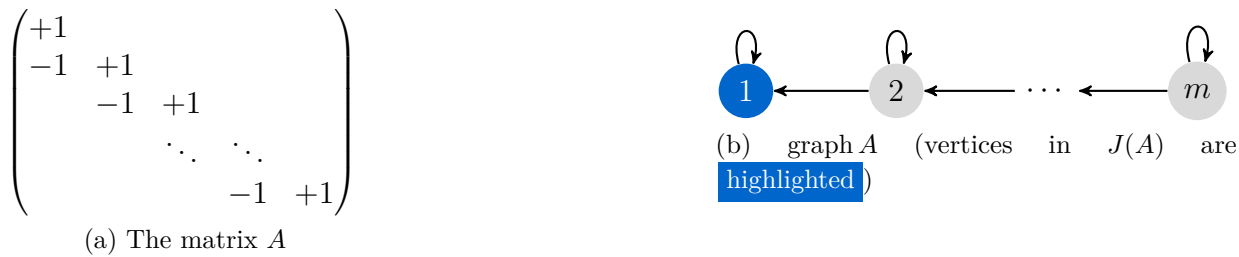
\begin{figure}
  \begin{centering}
    \subfloat[The matrix $A$]{
      \begin{centering}
        $
        \begin{pmatrix}+1\\
          -1 & +1\\
          & -1 & +1\\
          &  & \ddots & \ddots\\
          &  &  & -1 & +1
        \end{pmatrix}$
        \par
      \end{centering}
    }\hfill{}\subfloat[$\protect\graph A$ (vertices in $J(A)$ are
    \colorbox{mylinkcolor}{\textcolor{white}{highlighted}})]{
      \begin{tikzpicture}[node distance=2cm]
        \node [graph node, fill=mylinkcolor, text=white] (1) {1};
        \node [graph node, right of=1] (2) {2};
        \node [right of=2] (dots) {$\cdots$};
        \node [graph node, right of=dots] (m) {$m$};
        \draw[graph edge, ->, loop above] (1) to (1);
        \draw[graph edge, ->, loop above] (2) to (2);
        \draw[graph edge, ->, loop above] (m) to (m);
        \draw[graph edge, ->] (2) to (1);
        \draw[graph edge, ->] (dots) to (2);
        \draw[graph edge, ->] (m) to (dots);
      \end{tikzpicture}

    }
    \par
  \end{centering}
  \caption{\label{fig:wcdd_example}An example of a w.c.d.d. matrix
  and its graph}
\end{figure}

We now define the index of \emph{connectivity} associated with a square
complex w.d.d. matrix $A$ as
\[
  \conn A\coloneqq\max\left(0,\sup_{i\notin J(A)}\left\{ \infd_{p\in
  P_{i}(A)}\left|p\right|\right\} \right)
\]
(compare this with the index of \emph{contraction} $\widehat{\conn}$
defined in \eqref{eq:index_of_connectivity}). The lemma below is
a trivial consequence of the definitions above and as such requires
no proof.
\begin{lem}
  \label{lem:wcdd_iff_index_of_connectivity_is_finite}A square complex
  w.d.d. matrix $A$ is w.c.d.d. if and only if $\conn A$ is finite.
\end{lem}

We are now able to establish a duality between w.d.d. L-matrices (or
more accurately, $\operatorname{L}_{0}$-matrices) and substochastic
matrices that, as we will see, connects the nonsingularity of the
former to the convergence of the latter.
\begin{lem}
  \label{lem:link_between_wcdd_and_submarkov}Let $A\coloneqq(a_{ij})$
  be an $n\times n$ w.d.d. $\operatorname{L}_{0}$-matrix and
  $D\coloneqq(d_{ij})$
  be an $n\times n$ diagonal matrix whose diagonal entries are positive
  and satisfy $d_{ii}\leq1/a_{ii}$ for each $i$ such that $a_{ii}\neq0$.
  Then, $B\coloneqq I-DA$ is substochastic and
  \begin{equation}
    \conn A=\widehat{\conn}B.\label{eq:duality}
  \end{equation}

  Conversely, let $B$ be an $n\times n$ substochastic matrix and $D$
  be an $n\times n$ diagonal matrix whose diagonal entries are positive.
  Then, $A\coloneqq D(I-B)$ is a w.d.d. $\operatorname{L}_{0}$-matrix
  and \eqref{eq:duality} holds.
\end{lem}
\begin{proof}
  We prove only the first claim, the converse being handled similarly.

  Let $A$ and $B\coloneqq I-DA$ be given as in the lemma statement.
  To simplify notation, denote by $a_{ij}$ and $b_{ij}$ the elements
  of $A$ and $B$. First, note that $b_{ii}=1-d_{ii}a_{ii}\geq0$ and
  $b_{ij}=-d_{ii}a_{ij}\geq0$ whenever $i\neq j$. Since
  \[
    \sum_{j}b_{ij}=1-\sum_{j}d_{ii}a_{ij}=1-d_{ii}\left(a_{ii}-\sum_{j\neq
    i}\left|a_{ij}\right|\right)\leq1,
  \]
  it follows that $B$ is substochastic and $J(A)=\hat{J}(B)$. Letting
  $\graph A\equiv(V,E)$ and $\graph B\equiv(V^{\prime},E^{\prime})$,
  note that $V=V^{\prime}$ and, for $i\neq j$,
  \[
    (i,j)\in E\iff(i,j)\in E^{\prime}.
  \]
  Thus $\graph A$ and $\graph B$ differ only in their self-loops, which
  do not affect the existence or minimum length of a walk from a vertex
  to $J(A)=\hat{J}(B)$. As
  a result of these facts, \eqref{eq:duality} follows immediately.
\end{proof}
\begin{example}
  \label{exa:point_jacobi_matrix}Let $A\coloneqq(a_{ij})$ be a square
  w.d.d. L-matrix of order $n$ and
  \[
    B_{A}\coloneqq I-\diag(a_{11},\ldots,a_{nn})^{-1}A
  \]
  denote the point Jacobi matrix associated with $A$ (cf.
  \cite[Chapter 3]{MR1753713}).
  By the previous results, $A$ is w.c.d.d. if and only if $\conn
  A=\widehat{\conn}B_{A}$
  is finite.

  Note that the substochastic matrix in \prettyref{fig:submarkov_example}
  is the point Jacobi matrix associated with the w.d.d. L-matrix in
  \prettyref{fig:wcdd_example}.
\end{example}
We now restate and prove characterization \eqref{eq:characterization}
from the introduction.
\begin{thm}
  \label{thm:characterization}The following are equivalent:
  \begin{enumerate}[label=(\roman*), ref=(\emph{\roman*})]
    \item \label{enu:characterization_1}$A$ is a nonsingular w.d.d. M-matrix.
    \item \label{enu:characterization_2}$A$ is a nonsingular w.d.d. L-matrix.
    \item \label{enu:characterization_3}$A$ is a w.c.d.d. L-matrix.
  \end{enumerate}
\end{thm}
Since a nonsingular w.d.d. $\operatorname{L}_{0}$-matrix must be
an L-matrix, we can safely replace all occurrences of ``L-matrix''
with ``$\operatorname{L}_{0}$-matrix'' in the above theorem without
affecting its validity (recall that any w.c.d.d. matrix is nonsingular
\cite{MR0332820}).
\begin{proof}
  \ref{enu:characterization_1}$\implies$\ref{enu:characterization_2}
  follows from \prettyref{prop:an_M_matrix_is_an_L_matrix} while
  \ref{enu:characterization_3}$\implies$\ref{enu:characterization_1}
  is established in \cite[Theorem 2.2]{MR0162367}. We prove below the
  claim \ref{enu:characterization_2}$\implies$\ref{enu:characterization_3}.

  Let $A\coloneqq(a_{ij})$ be a nonsingular w.d.d. L-matrix of order
  $n$. Then, the associated point Jacobi matrix $B_{A}$ is substochastic
  and $I-B_{A}$ is nonsingular since
  \[
    I-B_{A}=\diag(a_{11},\ldots,a_{nn})^{-1}A.
  \]
  \prettyref{cor:eventual_contractivity} and
  \prettyref{lem:link_between_wcdd_and_submarkov}
  imply that $\conn A=\widehat{\conn}B_{A}$ is finite. Therefore, by
  \prettyref{lem:wcdd_iff_index_of_connectivity_is_finite}, $A$ is
  w.c.d.d.
\end{proof}
\begin{rem}
  Instead of calling upon the results of \cite{MR0162367}, it is also
  possible to prove
  \ref{enu:characterization_3}$\implies$\ref{enu:characterization_1}
  of \prettyref{thm:characterization} directly by using arguments involving
  the index of contraction. In particular, let $A$ be a w.c.d.d. L-matrix
  of order $n$. Then, by
  \prettyref{lem:wcdd_iff_index_of_connectivity_is_finite}
  and \prettyref{lem:link_between_wcdd_and_submarkov}, the associated
  point Jacobi matrix $B_{A}$ is substochastic with
  $\widehat{\conn}B_{A}=\conn A$
  finite. By \prettyref{cor:eventual_contractivity}, $B_{A}$ is convergent
  and hence the Neumann series $I+B_{A}+B_{A}^{2}+\cdots$ for the inverse
  of $I-B_{A}$ converges to a matrix whose entries are nonnegative.
  Therefore, $A$ is monotone by \prettyref{def:collatz}, and hence
  a nonsingular M-matrix by \prettyref{prop:M_matrix_monotone}.
\end{rem}
An immediate consequence of \prettyref{thm:characterization}, which
can be considered an analogue of \prettyref{cor:irreducibly_substochastic},
is given below.
\begin{cor}
  A square irreducible w.d.d. L-matrix $A$ is a nonsingular M-matrix
  if and only if $J(A)$ is nonempty.
\end{cor}
While the reverse direction in the above result is well-known
\cite[Corollary 3.20]{MR1753713},
we are not aware of a reference for the forward direction.

\section{\label{sec:computation}Computing the index of contraction}

In this section, we present a procedure to compute the index of contraction
$\widehat{\conn}B$ of a square substochastic matrix $B$ and show that it
is robust in the presence of inexact (i.e., floating point) arithmetic.

By the results of the previous section, such a procedure can also
be used to determine if an arbitrary w.d.d. matrix $A$ is a nonsingular
M-matrix as follows. If $A$ is not a square L-matrix, it is trivially
not a nonsingular M-matrix (\prettyref{prop:an_M_matrix_is_an_L_matrix}).
Otherwise, we can check the finitude of the index of contraction of
its associated point Jacobi matrix $B_{A}$ to determine whether or
not $A$ is a nonsingular M-matrix (recall \prettyref{exa:point_jacobi_matrix}
and \prettyref{thm:characterization}).

\subsection{The procedure}

Before we can describe the procedure, we require the notion of a vertex
contraction (a.k.a. vertex identification), a generalization of the
well-known notion of edge contraction from graph theory.
\begin{defn}
  Let $G\equiv(V,E)$ be a graph, $W\subset V$, $w$ denote a new vertex
  (i.e., $w\notin V$), and $f$ be a function which maps every vertex
  in $V\setminus W$ to itself and every vertex in $W$ to $w$ (i.e.,
  $f|_{V\setminus W}=\operatorname{id}_{V\setminus W}$ and $f|_{W}(\cdot)=w$).
  The vertex contraction of $G$ with respect to $W$ is a new graph
  $G^{\prime}\equiv(V^{\prime},E^{\prime})$ where
  $V^{\prime}\coloneqq(V\setminus W)\cup\{w\}$
  and $E^{\prime}\coloneqq\{(f(i),f(j))\colon(i,j)\in E\}$.
\end{defn}
An overview of the procedure for computing the index of contraction
for an arbitrary square substochastic matrix $B$ is given below:
\begin{enumerate}[label=(\emph{\arabic*})]
  \item \label{enu:steps_1}Obtain the vertex contraction of $\graph B$
    with respect to $\hat{J}(B)$. Label the new vertex in the contraction
    $w=0$ and the new vertex set $V^{\prime}$. Note that
    $V^{\prime}=\hat{J}(B)^{\complement}\cup\{0\}$
    (recall that the superscript $\complement$ denotes complement).
  \item \label{enu:steps_2}Reverse all edges in the resulting graph.
  \item \label{enu:steps_3}In the resulting graph, find the shortest distances
    $d(i)$ from the new vertex $0$ to all vertices $i\in V^{\prime}$
    by a breadth-first search (BFS) starting at $0$. It is understood
    that $d(0)=0$ and that if $i$ is unvisited in the BFS, $d(i)=\infty$.
  \item \label{enu:steps_4}Return $\max_{i\in V^{\prime}}d(i)$.
\end{enumerate}
That this procedure terminates is trivial (BFS is performed on a graph
with finitely many vertices). As for the correctness of the procedure,
it is easy to verify that
\[
  d(i)=\inf_{p\in\hat{P}_{i}(B)}|p|\text{ for }i\notin\hat{J}(B)
\]
so that
$\widehat{\conn}B=\max(0,\sup_{i\notin\hat{J}(B)}d(i))=\max_{i\in
V^{\prime}}d(i)$.
\begin{rem}
  \label{rem:useless_edges}Since BFS does not revisit vertices, the
  correctness of the procedure is unaffected if $\graph B$ is preprocessed
  to remove self-loops (i.e., edges of the form $i\rightarrow i$) and
  edges of the form $i\rightarrow j$ with $i\in\hat{J}(B)$.
\end{rem}
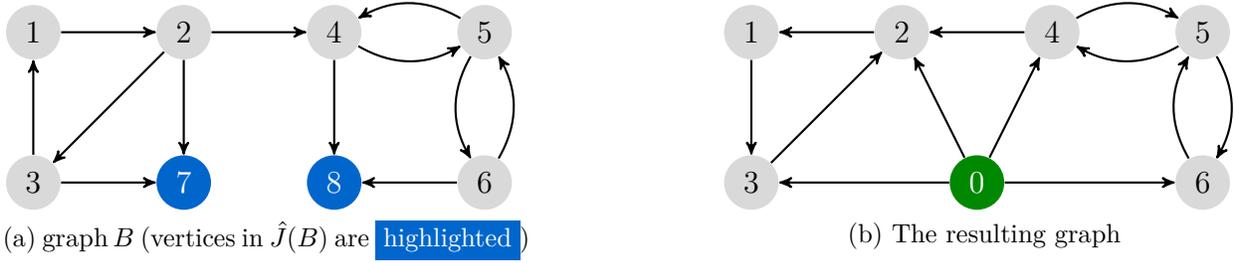
\begin{figure}
  \begin{centering}
    \subfloat[$\protect\graph B$ (vertices in $\hat{J}(B)$ are
    \colorbox{mylinkcolor}{\textcolor{white}{highlighted}})]{
      \begin{tikzpicture}[node distance=2cm]
        \node [graph node] (1) {1};
        \node [graph node, right of=1] (2) {2};
        \node [graph node, below of=1] (3) {3};
        \node [graph node, fill=mylinkcolor, text=white, right of=3] (7) {7};
        \node [graph node, fill=mylinkcolor, text=white, right of=7] (8) {8};
        \node [graph node, above of=8] (4) {4};
        \node [graph node, right of=4] (5) {5};
        \node [graph node, below of=5] (6) {6};
        \draw[graph edge, ->] (1) to (2);
        \draw[graph edge, ->] (2) to (3);
        \draw[graph edge, ->] (2) to (4);
        \draw[graph edge, ->] (2) to (7);
        \draw[graph edge, ->] (3) to (1);
        \draw[graph edge, ->] (3) to (7);
        \draw[graph edge, ->] (4) to [bend right] (5);
        \draw[graph edge, ->] (4) to (8);
        \draw[graph edge, ->] (5) to [bend right] (4);
        \draw[graph edge, ->] (5) to [bend right] (6);
        \draw[graph edge, ->] (6) to [bend right] (5);
        \draw[graph edge, ->] (6) to (8);
        %\draw[graph edge, ->] (7) to [bend right] (8);
        %\draw[graph edge, ->] (8) to [bend right] (7);
      \end{tikzpicture}

    }\hfill{}\subfloat[The resulting graph]{
      \begin{tikzpicture}[node distance=2cm]
        \node [graph node] (1) {1};
        \node [graph node, right of=1] (2) {2};
        \node [graph node, below of=1] (3) {3};
        \node [graph node, right of=2] (4) {4};
        \node [graph node, fill=mycitecolor, text=white, below of=2,
        xshift=1cm] (0) {0};
        \node [graph node, right of=4] (5) {5};
        \node [graph node, below of=5] (6) {6};
        \draw[graph edge, ->] (2) to (1);
        \draw[graph edge, ->] (3) to (2);
        \draw[graph edge, ->] (4) to (2);
        \draw[graph edge, ->] (0) to (2);
        \draw[graph edge, ->] (1) to (3);
        \draw[graph edge, ->] (0) to (3);
        \draw[graph edge, ->] (5) to [bend left] (4);
        \draw[graph edge, ->] (0) to (4);
        \draw[graph edge, ->] (4) to [bend left] (5);
        \draw[graph edge, ->] (6) to [bend left] (5);
        \draw[graph edge, ->] (5) to [bend left] (6);
        \draw[graph edge, ->] (0) to (6);
      \end{tikzpicture}

    }
    \par
  \end{centering}
  \caption{\label{fig:step_1_and_2_example}Steps \ref{enu:steps_1}
    and \ref{enu:steps_2}
  applied to an example}
\end{figure}

\begin{algorithm}
  \begin{flushleft}
    \footnotesize
    \hspace*{\algorithmicindent} \textbf{Input:} a square
    substochastic matrix $B \coloneqq (b_{ij})_{1 \leq i, j \leq n}$
    of order $n$ \\
    \hspace*{\algorithmicindent} \textbf{Output:} $\widehat{\protect\conn}B$
  \end{flushleft}
  \begin{multicols}{2}
    \footnotesize
    \begin{algor}[1]
    \item [{{*}}] \emph{// Find all rows in $\hat{J}(B)$}
    \item [{{*}}] $s\gets0$
    \item [{{*}}] $S[1,\ldots,n]\gets\text{\textbf{new} array of bools}$
    \item [{forall}] rows $i$
      \begin{algor}[1]
      \item [{{*}}] $t\gets0$
      \item [{forall}] cols $j$ s.t. $b_{ij}\neq0$\label{code:sparse_loop_1}
        \begin{algor}[1]
        \item [{{*}}] $t\gets t+b_{ij}$\label{code:sum}
        \end{algor}
      \item [{endfor}]~
      \item [{if}] $t<1$\label{code:compare}
        \begin{algor}[1]
        \item [{{*}}] $s\gets s+1$
        \item [{{*}}] $S[i]\gets\TRUE$\emph{ // $i\in\hat{J}(B)$}
        \end{algor}
      \item [{else}]~
        \begin{algor}[1]
        \item [{{*}}] $S[i]\gets\FALSE$\emph{ // $i\notin\hat{J}(B)$}
        \end{algor}
      \item [{endif}]~
      \end{algor}
    \item [{endfor}]~
    \item [{{*}}]~
    \item [{{*}}] \emph{// Find neighbours of each vertex (ignoring extraneous
      edges as per \prettyref{rem:useless_edges})}
    \item [{{*}}] $N[0,\ldots,n]\gets\text{\textbf{new }array of lists}$
    \item [{forall}] rows $i$ s.t. $S[i]=\FALSE$
      \begin{algor}[1]
      \item [{forall}] cols $j\neq i$ s.t.
        $b_{ij}\neq0$\label{code:sparse_loop_2}
        \begin{algor}[1]
        \item [{if}] $S[j]=\TRUE$
          \begin{algor}[1]
          \item [{{*}}] $N[0]$.add($i$)
          \end{algor}
        \item [{else}]~
          \begin{algor}[1]
          \item [{{*}}] $N[j]$.add($i$)
          \end{algor}
        \item [{endif}]~
        \end{algor}
      \item [{endfor}]~
      \end{algor}
    \item [{endfor}]~
    \item [{{*}}]~
    \item [{{*}}] \emph{// Perform BFS starting at $0$}
    \item [{{*}}] $\operatorname{result}\gets0$
    \item [{{*}}] $Q\gets\text{\textbf{new }queue}$
    \item [{{*}}] $Q$.enqueue($(0,0)$)
    \item [{while}] $Q$ is not empty
      \begin{algor}[1]
      \item [{{*}}] $(j,d)\gets Q$.dequeue()
      \item [{{*}}] $\operatorname{result}\gets\max(\operatorname{result},d)$
      \item [{forall}] $i$ in $N[j]$ s.t. $S[i]=\FALSE$
        \begin{algor}[1]
        \item [{{*}}] $s\gets s+1$
        \item [{{*}}] $S[i]\gets\TRUE$
        \item [{{*}}] $Q$.enqueue($(i,d+1)$)
        \end{algor}
      \item [{endfor}]~
      \end{algor}
    \item [{endwhile}]~
    \item [{{*}}]~
    \item [{if}] $s=n$
      \begin{algor}[1]
      \item [{{*}}] $\widehat{\conn}B\gets\operatorname{result}$
      \end{algor}
    \item [{else}]~
      \begin{algor}[1]
      \item [{{*}}] $\widehat{\conn}B\gets\infty$
      \end{algor}
    \item [{endif}]~
    \end{algor}
  \end{multicols}

  \vspace{6pt}

  \caption{\label{alg:code}Computing the index of contraction of a
    square substochastic
  matrix}
\end{algorithm}

\prettyref{alg:code} gives precise pseudocode for steps \ref{enu:steps_1}
to \prettyref{enu:steps_4}. Without loss of generality, it is assumed
that the input matrix is square (the rectangular case is obtained
by a few trivial additions to the code). The pseudocode makes use
of the \emph{list} and \emph{queue} data structures (see, e.g.,
\cite[Chapter 10]{cormen2001introduction}).
The operation $L$.add($x$) appends the element $x$ to the list
$L$. The operation $Q$.enqueue($x$) adds the element $x$ to the
back of the queue $Q$. The operation $Q$.dequeue() removes and returns
the element at the front of the queue $Q$.

It is obvious that if the input to \prettyref{alg:code} is a dense
matrix of order $n$, $\Theta(n^{2})$ operations are required. Suppose
instead that we restrict our inputs to matrices $B\coloneqq(b_{ij})$
that are \emph{sparse} in the sense that
$\operatorname{nnz}\coloneqq\max_{i}|\{j\colon b_{ij}\neq0\}|$,
the maximum number of nonzero entries per row, is bounded independent
of $n$ (i.e., $\operatorname{nnz}=\Theta(1)$ as $n\rightarrow\infty$).
If the matrices are stored in an appropriate format (e.g., compressed
sparse row (CSR) format, Ellpack-Itpack, etc. \cite{MR1990645}),
the loops on lines \ref{code:sparse_loop_1} and \ref{code:sparse_loop_2}
require only a constant number of iterations for each fixed $i$.
In this case, $\Theta(n)$ operations are required. An obvious generalization
of this fact is that if $\operatorname{nnz}=O(f(n))$, $O(nf(n))$
operations are required.

\subsection{Floating point arithmetic considerations}

The loop on line \ref{code:sparse_loop_1} of \prettyref{alg:code}
computes the $i$-th row-sum of the substochastic matrix $B\coloneqq(b_{ij})$.
In the presence of floating point arithmetic, the operation $t+b_{ij}$
on line \ref{code:sum} can introduce error into calculations. In
order to analyze this error, we take the standard model of floating
point arithmetic in which floating point addition introduces error
proportional to the size of the result:
\begin{equation}
  fl(x+y)=\left(x+y\right)\left(1+\delta_{x,y}\right)\text{ where
  }\left|\delta_{x,y}\right|\leq\epsilon.\label{eq:floating_point}
\end{equation}
$\epsilon>0$ is a machine-dependent constant (often referred to as
machine epsilon) which gives an upper bound on the relative error
due to rounding. In performing our analyses, we make the standard
assumptions that the order $n$ of the input matrix $B$ satisfies
$n\epsilon\leq1$ \cite{MR1223274} and that the entries of $B$ are
floating point numbers.

Let $i$ be a row of $B$ with $m(i) \geq 1$ nonzero entries denoted
$b_{ij_1}, \ldots, b_{ij_{m(i)}}$.
A floating point implementation of the loop on line \ref{code:sparse_loop_1}
is represented by the recurrence $S_{\ell}\coloneqq
fl(S_{\ell-1}+b_{ij_\ell})$ for $\ell = 2, \ldots, m(i)$
with initial condition $S_{1}\coloneqq b_{ij_1}$. Letting
$\gamma_{k}\coloneqq k\epsilon/(1-k\epsilon)$,
this direct implementation has an error bound of \cite[Eq. (2.6)]{MR1223274}
\begin{equation}
  \left|S_{m(i)}-\sum_{j}b_{ij}\right|\leq\gamma_{\operatorname{nnz}-1}\sum_{j}b_{ij}\leq\gamma_{\operatorname{nnz}-1}.\label{eq:error}
\end{equation}
Recall that $\operatorname{nnz}$ is the maximum number of nonzero
entries per row of the matrix $B$. If the matrix $B$ is sparse (i.e.,
$\operatorname{nnz}=\Theta(1)$ as $n\rightarrow\infty$), we obtain
\[
  \gamma_{\operatorname{nnz}-1}=(\operatorname{nnz}-1)\epsilon+O(\epsilon^{2})\text{
  as }\epsilon\rightarrow0
\]
by the power series representation of $\gamma_{k}$. In this case,
for each $i$, the \emph{absolute} error in computing $\sum_{j}b_{ij}$
is independent of $n$.

Note that if the \emph{exact} value of $\sum_{j}b_{ij}$ is close
to $1$, the comparison $t<1$ on line \ref{code:compare} may return
either a false-positive or a false-negative. Motivated by \eqref{eq:error},
an implementation of \prettyref{alg:code} should use instead the
condition $t<1-\operatorname{tol}$ where $\operatorname{tol}$ is
a small constant strictly larger than $\gamma_{\operatorname{nnz}-1}$
to preclude the possibility that the condition evaluates to true when
the \emph{exact} value of $\sum_{j}b_{ij}$ is $1$ (for simplicity,
we assume $1-\operatorname{tol}$ has a precise floating point representation).
Then, the error bound \eqref{eq:error} and discussion above yield
the accuracy result below.
\begin{lem}
  \label{lem:error}Let $B$ be a square substochastic matrix with at most
  $\operatorname{nnz} \geq 1$ nonzero entries per row. Denoting by
  $(\widehat{\conn}B)_{fl}$
  the quantity computed by \prettyref{alg:code} under the standard
  model of floating point arithmetic \eqref{eq:floating_point} and
  with condition $t<1$ replaced by $t<1-\operatorname{tol}$ where
  $\operatorname{tol}>\gamma_{\operatorname{nnz}-1}$, the following
  results hold:
  \begin{enumerate}[label=(\roman*), ref=(\emph{\roman*})]
    \item if $\widehat{\conn}B=\infty$, then
      $(\widehat{\conn}B)_{fl}=\widehat{\conn}B$.
    \item if $\widehat{\conn}B\neq\infty$ and
      $\sum_{j}b_{ij}\leq1-2\operatorname{tol}$
      for $i\in\hat{J}(B)$, then $(\widehat{\conn}B)_{fl}=\widehat{\conn}B$.
  \end{enumerate}
\end{lem}
\begin{rem}
  If $B$ is not sparse, the error \eqref{eq:error} depends on $n$.
  In this case, one should substitute the naïve summation outlined by
  the loop on line \ref{code:sparse_loop_1} for a more scalable algorithm,
  such as Kahan's summation algorithm, whose absolute error in approximating
  $\sum_{j}b_{ij}$, is
  $(2\epsilon+O(n\epsilon^{2}))\sum_{j}b_{ij}\leq2\epsilon+O(n\epsilon^{2})$
  \cite[Eq. (3.11)]{MR1223274}, which is independent of $n$ due to
  the assumption $n\epsilon\leq1$. We can obtain an analogue of
  \prettyref{lem:error}
  under Kahan summation by choosing $\operatorname{tol}$ appropriately.
\end{rem}
Note that \prettyref{lem:error} suggests that $(\widehat{\conn}B)_{fl}$
may be infinite even when $\widehat{\conn}B$ is finite. Fortunately,
as demonstrated in the next example, this occurs only if the matrix
$B$ is ``nearly nonconvergent'' (i.e., $\rho(B)=1-\epsilon_{0}$
where $\epsilon_{0}>0$ is close to zero). This error may even be
considered desirable behaviour since a nearly nonconvergent matrix
may not be convergent in the presence of floating point error.
\begin{example}
  Consider the $n\times n$ matrix
  \[
    B_{\nu}\coloneqq
    \begin{pmatrix} & 1\\
      &  & 1\\
      &  &  & \ddots\\
      &  &  &  & 1\\
      &  &  &  &  & 1\\
      \nicefrac{1}{n}-\nu & \nicefrac{1}{n} & \nicefrac{1}{n} &
      \cdots & \nicefrac{1}{n} & \nicefrac{1}{n}
    \end{pmatrix}
  \]
  where $0<\nu\leq1/n$. Note that even though $\widehat{\conn}B_{\nu}=n-1$
  independent of the value of $\nu$, $\rho(B_{\nu})\rightarrow1$ as
  $\nu\rightarrow0$. When $\nu$ is very close to zero, floating point
  error may cause \prettyref{alg:code} to erroneously determine that
  $\hat{J}(B_{\nu})$ is empty and thereby mistakenly conclude that
  the index of contraction is infinite.
\end{example}
We close this section by discussing stability. The test in \cite{MR2047092},
which determines if an arbitrary matrix is a nonsingular M-matrix,
uses a modified Gaussian elimination procedure. As such, to establish
numerical stability, the author proves that the \emph{growth factor}
(see the definition in \cite{MR3024913}) of the test is bounded by
the order of the input matrix \cite[Theorem 3.1]{MR2047092}. In our
case, the floating point error made in computing $\sum_{j}b_{ij}$
has no bearing on the error made in computing $\sum_{j}b_{i^{\prime}j}$
for distinct rows $i$ and $i^{\prime}$. That is, floating point
errors do not propagate from row to row. Moreover, as demonstrated
in the previous paragraphs, the error made in computing each row-sum
can be bounded by a constant (without any additional effort in the
sparse case, and with, e.g., Kahan summation in the dense case). As
such, we conclude that \prettyref{alg:code} is stable in the sense
that it does not involve numbers that grow large due to floating point
error.

\section{\label{sec:experiments}Numerical experiments}

In this section, we compare the efficiency of \emph{our test} described
at the beginning of \prettyref{sec:computation} to \emph{Peña's test}
detailed in \cite{MR2047092}. To minimize bias, we run the tests
on randomly sampled matrices (sampled according to the procedure in
\prettyref{app:sampling}).

We run the tests on matrices whose maximum number of nonzeros per
row ($\operatorname{nnz}$) are $6$, $12$, $24$, and $48$. We
employ two versions of our test: a \emph{sparse} version, in which
the matrices are stored in compressed sparse row (CSR) format, and
a \emph{dense} version, in which the matrices are two-dimensional
arrays. All tests are performed on an Intel Xeon E5440 2.83GHz CPU.
The average time to process a randomly sampled matrix is shown in
\prettyref{fig:experiments} (error bars are omitted as even the 99\%
confidence interval is too small to be visible). We mention that in
terms of accuracy, the tests produced the same results on all randomly
sampled matrices (\prettyref{fig:probability}).

\begin{figure}
  \begin{centering}
    \hspace{2em}\includegraphics[scale=0.75]{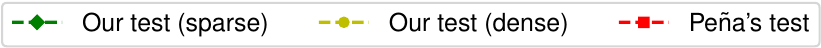}
    \par
  \end{centering}
  \begin{centering}
    \subfloat[Timing results ($\log_{10}$ scale on time
    axis)\label{fig:experiments}]{
      \begin{centering}
        \includegraphics[scale=0.75]{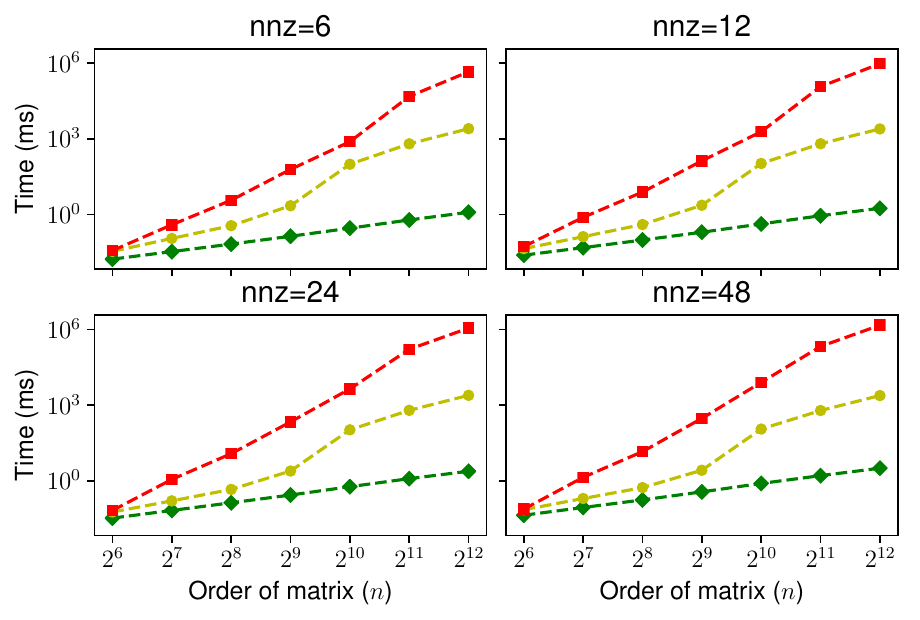}
        \par
      \end{centering}
    }
    \par
  \end{centering}
  \centering{}\subfloat[Probability that a randomly sampled matrix is
    a nonsingular M-matrix
  (99\% confidence intervals shown)\label{fig:probability}]{
    \begin{centering}
      \includegraphics[scale=0.75]{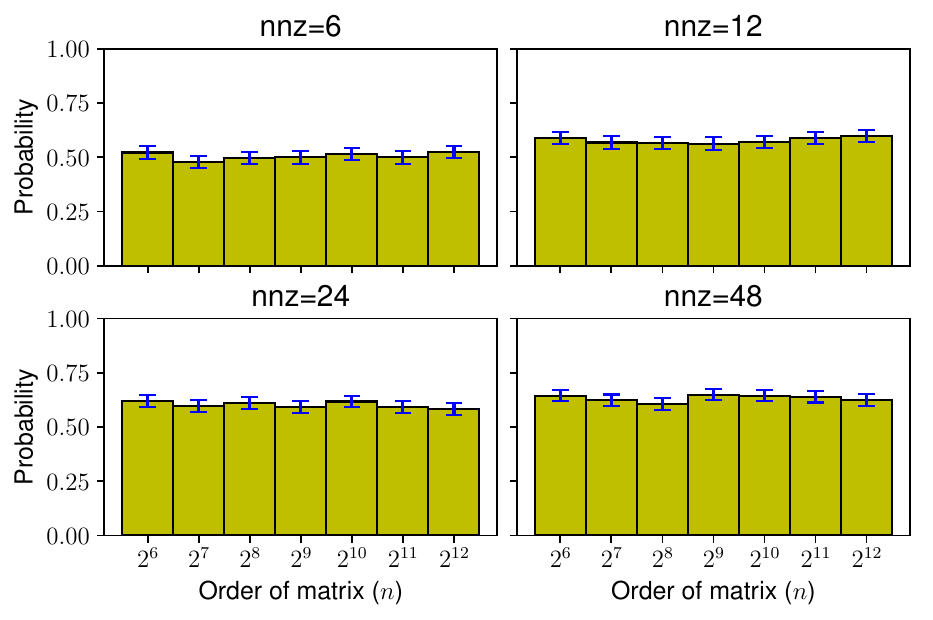}
      \par
    \end{centering}
  }
\end{figure}

\prettyref{fig:experiments} suggests that our test outperforms Peña's.
Even for the experiments involving the $1024\times1024$ sparse matrices
(small by most scientific computing standards), our sparse implementation
executes on the order of \emph{tenths of milliseconds} while Peña's
test executes on the order of \emph{seconds}.

\begin{description}
  \item [{Acknowledgments}] The author thanks Edward Cheung (University of
    Waterloo) for discussions on M-matrices, a careful review of this
    document, and (more importantly) his unfaltering friendship.
\end{description}
\bibliographystyle{plainnat}
\bibliography{main}

\appendix

\section{\label{app:sampling}Sampling procedure}

This appendix details the procedure (employed in the numerical experiments
of \prettyref{sec:experiments}) used to randomly sample w.d.d.
$\operatorname{L}_{0}$-matrices.
The procedure, for which pseudocode is given in \prettyref{alg:sampling},
works by sampling a matrix $B\coloneqq(b_{ij})_{1\leq i,j\leq n}$
from the space of substochastic matrices and returning $I-B$, which
is a w.d.d. $\operatorname{L}_{0}$-matrix by
\prettyref{lem:link_between_wcdd_and_submarkov}.

\begin{algorithm}[h]
  \begin{flushleft}
    \footnotesize
    \hspace*{\algorithmicindent} \textbf{Input:} positive integers
    $n$ and $\operatorname{nnz} \leq n$ \\
    \hspace*{\algorithmicindent} \textbf{Output:} matrix $A$
  \end{flushleft}
  \begin{multicols}{2}
    \footnotesize
    \begin{algor}[1]
    \item [{{*}}] \emph{// Initialize zero matrix}
    \item [{{*}}] $B\equiv(b_{ij})\gets0$
    \item [{{*}}]~
    \item [{for}] $i$ from $1$ to $n$
      \begin{algor}[1]
      \item [{{*}}] \emph{// Determine the number of nonzero entries $m$ in
        row $i$}
      \item [{{*}}] {$m\sim\operatorname{Unif}\{1,\ldots,\operatorname{nnz}\}$}
      \item [{{*}}]~
      \item [{\emph{{*}}}] \emph{// Determine the row-sum of row $i$ (less than
        one with probability $1/n$)}
      \item [{{*}}] $u\sim\operatorname{Unif}[0,1]$
      \item [{if}] $u<1/n$
        \begin{algor}[1]
        \item [{{*}}] $s\sim\operatorname{Unif}[0,1]$
        \end{algor}
      \item [{else}]~
        \begin{algor}[1]
        \item [{{*}}] $s\gets1$
        \end{algor}
      \item [{endif}]~
      \item [{{*}}]~
      \item [{\emph{{*}}}] \emph{// Determine the indices $j_{k}$ for which
          $b_{ij_{k}}$ is nonzero by uniformly sampling $\{1,\ldots,n\}$ without
        replacement}
      \item [{{*}}] {$\mathcal{A}\gets\{1,\ldots,n\}$}
      \item [{{*}}] $j_{1}\sim\operatorname{Unif}\mathcal{A}$
      \item [{for}] $k$ from $2$ to $m$
        \begin{algor}[1]
        \item [{{*}}] {$\mathcal{A}\gets\mathcal{A}\setminus\{j_{k-1}\}$}
        \item [{{*}}] $j_{k}\sim\operatorname{Unif}\mathcal{A}$
        \end{algor}
      \item [{endfor}]~
      \item [{{*}}]~
      \item [{\emph{{*}}}] \emph{// Determine the values of the nonzero entries
        in row $i$}
      \item [{{*}}] $(b_{ij_{1}},\ldots,b_{ij_{m}})\gets
        s\cdot\operatorname{Dir}(1,\ldots,1)$
      \end{algor}
    \item [{endfor}]~
    \item [{{*}}]~
    \item [{{*}}] \emph{// Make a w.d.d. $\operatorname{L}_{0}$-matrix from
      the substochastic matrix $B$}
    \item [{{*}}] $A\gets I-B$
    \end{algor}
  \end{multicols}

  \vspace{6pt}

  \caption{\label{alg:sampling}Sampling a matrix $A$ from the space of w.d.d.
  $\operatorname{L}_{0}$-matrices}
\end{algorithm}

We use $\operatorname{Unif}\Omega$ to denote a uniform distribution
on the sample space $\Omega$. For $\alpha\in\mathbb{R}^{m}$, we
use $\operatorname{Dir}\alpha$ to denote a Dirichlet distribution
of order $m$ with parameter $\alpha$. It is well-known that when
$\alpha$ is a vector whose entries are all one, $\operatorname{Dir}\alpha$
is a uniform distribution over the unit simplex in $\mathbb{R}^{m-1}$.
We use $x\sim\mathcal{D}$ to mean that $x$ is a sample drawn from
the distribution $\mathcal{D}$.

The inputs to the procedure are a positive integer $n$ corresponding
to the order of the output matrix and a positive integer
$\operatorname{nnz}\leq n$
corresponding to the maximum number of nonzero entries per row.

\section{Generalizing \prettyref{thm:eventual_contractivity}}

This appendix generalizes \prettyref{thm:eventual_contractivity}.
To present the generalization, we first extend our notion of walks:
\begin{defn}
  Let $(A_{n})_{n\geq1}$ be a sequence of compatible complex matrices
  (i.e., the product $A_{k}A_{k+1}$ is well defined for each $k$).
  \begin{enumerate}[label=(\emph{\roman*})]
    \item A walk in $(A_{n})_{n}$ is a nonempty finite sequence of
      edges $(i_{1},i_{2})$,
      $(i_{2},i_{3})$, $\ldots$, $(i_{\ell-1},i_{\ell})$ such that each
      $(i_{k},i_{k+1})$ is an edge in $\graph A_{k}$. The set of all walks
      in $(A_{n})_{n}$ is denoted $\walks(A_{1},A_{2},\ldots)$.
    \item For $p\in\walks(A_{1},A_{2},\ldots)$, $\head p$, $\last p$, and
      $|p|$ are defined in the obvious way.
  \end{enumerate}
\end{defn}
Note, in particular, that if we fix a square complex matrix $A$,
we are returned to the original definition of a walk given in
\prettyref{sec:matrix_families}
if we take $A_{n}\coloneqq A$ for all $n$.

It is also useful to generalize the sets $\hat{P}_{i}(\cdot)$ of
\prettyref{sec:matrix_families}. In particular, given a sequence
$(B_{n})_{n\geq1}$ of compatible substochastic matrices, let
\[
  \hat{P}_{i}(B_{1},B_{2},\ldots)\coloneqq\Bigl\{
    p\in\walks(B_{1},B_{2},\ldots)\colon\head p=i\text{ and }\last
  p\in\hat{J}(B_{|p|+1})\Bigr\}.
\]
We are now ready to give the generalization.
\begin{thm}
  Let $(B_{n})_{n\geq1}$ be a sequence of compatible substochastic
  matrices, $(C_{n})_{n\geq0}$ be defined by $C_{0}\coloneqq I$ and
  $C_{n}\coloneqq B_{1}B_{2}\cdots B_{n}$ whenever $n$ is a positive
  integer, and
  \[
    \widehat{\conn}(B_{1},B_{2},\ldots)\coloneqq\max\left(0,\sup_{i\notin\hat{J}(B_{1})}\left\{
    \infd_{p\in\hat{P}_{i}(B_{1},B_{2},\ldots)}\left|p\right|\right\} \right).
  \]
  If $\alpha\coloneqq\widehat{\conn}(B_{1},B_{2},\ldots)$ is finite,
  \[
    1=\Vert C_{0}\Vert_{\infty}=\cdots=\Vert
    C_{\alpha}\Vert_{\infty}>\Vert
    C_{\alpha+1}\Vert_{\infty}\geq\Vert C_{\alpha+2}\Vert_{\infty}\geq\cdots
  \]
  Otherwise,
  \[
    1=\Vert C_{1}\Vert_{\infty}=\Vert C_{2}\Vert_{\infty}=\cdots
  \]
\end{thm}
The proof of the above is nearly identical to that of
\prettyref{thm:eventual_contractivity},
requiring only a simple generalization of \prettyref{lem:shrinking_paths}.
However, in this general case, the finitude of the index of contraction
is no longer an indicator of convergence:
\begin{example}
  Let $B_{n}\coloneqq(1-1/2^{n})\in\mathbb{R}^{1\times1}$ for
  $n\geq1$, and let $(C_{n})_{n\geq0}$ be defined as above. Clearly,
  each matrix $B_{n}$ is convergent,
  but $\Vert C_{n}\Vert_{\infty}=\prod_{k=1}^{n}(1-1/2^{k})\nrightarrow0$
  as $n\rightarrow\infty$.
\end{example}
Moreover, even if each $B_{n}$ is itself convergent, it is still
possible that the index of contraction is infinite:
\begin{example}
  Let $(B_{n})_{n\geq1}$ be given by
  \[
    B_{n}\coloneqq\frac{1}{2}
    \begin{pmatrix}0 & 1+(-1)^{n}\\
      1-(-1)^{n} & 0
    \end{pmatrix}.
  \]
  Defining $(C_{n})_{n\geq0}$ as above, we find that
  \[
    C_{n}\coloneqq\frac{1}{2}
    \begin{pmatrix}0 & 0\\
      1-(-1)^{n} & 1+(-1)^{n}
    \end{pmatrix}\text{ for }n\geq1.
  \]
  That is, $\Vert C_{n}\Vert_{\infty}=1$ independent of $n$.
\end{example}
It is not hard to find interesting cases in which
$\widehat{\conn}(B_{1},B_{2},\ldots)$
is finite:
\begin{example}
  Let $(B_{n})_{n\geq1}$ be a sequence of square substochastic matrices
  of order $m$ satisfying the following properties:
  \begin{enumerate}[label=(\emph{\roman*})]
    \item $B_{1}$ is convergent.
    \item $\hat{J}(B_{1})=\hat{J}(B_{n})$ and $\graph B_{1}=\graph B_{n}$
      for all $n$.
  \end{enumerate}
  Then, $\widehat{\conn}(B_{1},B_{2},\ldots)<m$.
\end{example}

\end{document}